\documentclass{article}
\usepackage[french,english]{babel}
 \usepackage[T1]{fontenc}        
 \usepackage[latin1]{inputenc}   
 \usepackage{amssymb,amsfonts,amsthm,amsmath}
 \usepackage{graphicx}
 
 \newtheorem{theorem}{Theorem}[section]
 \newtheorem{prop}[theorem]{Proposition}
 
 \newtheorem{corollary}[theorem]{Corollary}
 \newtheorem{definition}{Definition}
 \newtheorem{lemma}[theorem]{Lemma}
 
 \title{The Newman phenomenon and Lucas sequence.}
 \author{Alexandre Aksenov}
 \date{}
 
\newcommand{\F}{\mathbb F_p}
\newcommand{\Fx}{\mathbb F_p^\times}
\newcommand{\R}{\mathbb R}
\newcommand{\N}{\mathbb N}
\newcommand{\Z}{\mathbb Z}
\newcommand{\C}{\mathbb C}
\newcommand{\Q}{\mathbb Q}
\newcommand{\ind}{1} 
\newcommand{\pl}{\oplus}
\newcommand{\Minus}{\frac{\phantom{xxx}}{\phantom{xxx}}\,}
\newcommand{\Tr}{\mathop{\rm Tr}}
\newcommand{\uo}{\bar{1}}
\renewcommand{\t}{\cdot}

\begin{document}
\maketitle

\begin {abstract}
This article gives an  alternative proof of the fact that $N_{\Q(\zeta)/\Q}(1-\zeta)=p$ where $p$ is an odd prime number  and $\zeta$ is a primitive $p$-th root of unity, 
and uses it to prove that $N_{\Q(\zeta)/\Q}(1+\zeta-\zeta^2)=L(p)$ the $p$-th Lucas number. It shows a relation between this result and a generalisation of the Newman phenomenon.
\end{abstract}

\vspace*{0.25cm}\noindent{\small {\bf Keywords}:
Newman phenomenon, M\"obius function, partitions of a set, Lucas numbers, cyclotomic fields.}

\section*{1. Introduction.}
\refstepcounter{section}
In 1969  D.J.Newman proved the following statement:

\begin{theorem}[Moser's conjecture]
Let $(t_n)$ be the Thue-Morse sequence i.e.
$$t_n=(-1)^{\text{\rm the number of '$1$' digits in the binary expansion of $n$}} \in \{1,-1\}. $$

Then all the sums $t_0$, $t_0+t_3$, $t_0+t_3+t_6$ (the sums of $N$ initial terms of the sequence $(t_{3n})$) are strictly positive. 
\end{theorem}

He also gave an estimation for the rate of growth of these sums, more precisely he proved that

$$K_1 N^\frac{\log 3}{\log 4} < \sum_
																{\begin{array}{c}
																	n<N\\
																	3|n	
																\end{array}} 
																							t_{n} < K_2 N^\frac{\log 3}{\log 4}$$
for two explicit constants $K_1$ and $K_2$.

The article \cite{Coquet} of J.Coquet  provides a more complete description of the structure of these sums via the formula
$$
\sum_
{\begin{array}{c}
n<N\\
3|n	
\end{array}} 
							t_{n} = F(\{\log_4 N\}) N^\frac{\log 3}{\log 4}+ \epsilon(N)
$$
where $\{x\}$ denotes the fractional part of $x\in\R$, $|\epsilon(N)|\leqslant 1$ therefore $\epsilon(N)$ is an error term, and $F$ is a continuous and nowhere differentiable function. This article gives also the values of the the maximum and the minimum of $F$, that are the best possible bounds $K_2$ and $K_1$ for $N$ large enough. 

Coquet's result has been generalised to arbitrary odd primes $p$ in \cite{GKS}:
\begin{theorem}[Theorem 5.1 from \cite{GKS}]
Let $p$ be an odd prime number. Denote 
$$\mathbf S_q(N)=(S_{q,i}(N)=\sum_{
\begin{array}{c}
0\leqslant n<N\\
n\equiv i\text{\rm\ mod }p 
\end{array}
} t(n)\ )_{i=0,1,\ldots,p-1}\in\C^p, $$ 
denote $\mathbf T$ the $p\times p$ matrix
$$\mathbf T=\left(
\begin{array}{ccccc}
0&0  &\ldots           &0     &1     \\
1&0  &                 &      &0     \\
 &1  &0\phantom{x}     &      &0     \\
 &(0)&\diagdown        &\ddots&\vdots\\
 &   &                 &1     &0
\end{array}\right)
$$
and 
$$\mathbf M=\prod_{m=0}^{s-1}(\mathbf I-\mathbf T^{2^m}) $$
where $s$ is the order of the element $2$ in the group $\Fx$.
Then the asymptotic behaviour of the sequence of vectors $\mathbf S_q(N)$ is described by
\begin{equation}
\label{GKS}
\mathbf S_p(N)=N^{\alpha_p}\tilde{\mathbf G}\left(\frac{\log N}{rs\log 2}\right)+\tilde{\mathbf E}(N)
\end{equation}
where:\begin{itemize}
\item $\tilde{\mathbf G}(x)=\left(\tilde G_i(x)\right)_{i=0,1,\dots,p-1}$ is a column vector whose components are continuous nowhere derivable periodic functions with period $1$;
\item $\alpha_p=\frac{\log \lambda_1}{rs\log 2}$ where $\lambda_1$ is the spectral radius of $\mathbf M$ and $r\in\{1,2,4\}$ is the smallest integer such that the $r$-th power of the biggest (in modulus) eigenvalue of $\mathbf M$ is real positive;
\item $\tilde{\mathbf E}(N)=O(N^{\beta_p})$ \\
where $\beta_p$ depends on the second biggest modulus $\lambda_2$ of an eigenvalue of  $\mathbf M$ by
$$\beta_p=\beta_p(\lambda_2)=\left\{\begin{array}{ll}
\frac{\log\lambda_2}{s\log 2}&\mbox{\em if $\lambda_2>1$},\\
0&\mbox{\em otherwise,}
\end{array}\right. $$
\end{itemize}
\end{theorem}
 
They did not calculate the best possible values for bounds $K_1$ and $K_2$ since the complexity of this calculation increases very fast with $p$ (see also \cite{Grabner} for a complete study of the case $p=5$). Remark that in fact, $r=2$ never happens as will be proved in Section $5$.

In this article we are going to further generalize the method of Coquet  to a large class of $b$-multiplicative sequences with arbitrary $b>1$ and challenge another problem explored 
in \cite{GKS}: give an algebraic description for the exponent $\alpha_p$ of $N$ in the formulae analogous to (\ref{GKS}). 

Recall that a sequence $f(n)$ of real or complex numbers is called $b$-multiplicative if it satisfies
$f(0)=1$ and  $$f(ab^k+c)=f(ab^k)f(c) $$
for all $a,k\in\N$ and $c<b^k$.

We will also denote by $\overline{c_lc_{l-1}\dots c_0}_{(b)}$, or simply $\overline{c_lc_{l-1}\dots c_0}$ when the base of the numeration system is clear, the nonnegative integer represented by the sequence $c_lc_{l-1}\dots c_0$ in base $b$, that is:
\[\overline{c_lc_{l-1}\dots c_0}_{(b)} = \sum_{i=0}^l c_ib^i. \]

\section*{2. $p$-rarefaction in multiplicative sequences.}
\refstepcounter{section}
Let $b>1$ be an integer and $\tau_n$ be a unit-circle-valued $b$-multiplicative sequence such that $\tau_{cb^k}=\tau_c$ for any $c,k\in\N$. As consequence $\tau$ can be defined  as follows: if $\overline{c_lc_{l-1}\dots c_0}$ is the $b$-ary expansion of an integer $n$ then
$$\tau_n=\prod_{i=0}^l \tau_{c_i} $$
and $\tau_0=1$; each of the complex numbers $\tau_1,\tau_2,\dots,\tau_{b-1}$ is of modulus one. We will also suppose that $\tau_n$ is not identically $1$.

Under these hypotheses the Theorem $2$ of \cite{Morgenbesser} can be applied, and it states that
$$\sum_{n=0}^{N-1}\tau_n = o(N). $$
This bound can be made more precise.

Let us put an additional notation : for each $c\in\{0,\dots,b\}$ we will denote $d(c)=\sum\limits_{i=0}^{c-1} \tau_i$.
Then for each positive integer $N$ represented in base $b$ by $\overline{c_lc_{l-1}\dots c_0}$ we get:

\begin{equation}
\label{Spsi}
\begin{array}{cl}
\sum\limits_{n=0}^{N-1} \tau_n&= \sum\limits_{n=0}^{\overline{c_l00\dots0}-1}\tau_n+ \sum\limits_{n=\overline{c_l00\dots0}}^{\overline{c_lc_{l-1}0\dots0}-1}\tau_n+\ldots+
\sum\limits_{n=\overline{c_lc_l-1\dots c_10}}^{N-1}\tau_n\\[2mm]
&=\left(\sum\limits_{c=0}^{c_l-1}\tau_c\right)\left(\sum\limits_{c=0}^{b-1}\tau_c\right)^l+\tau_{c_l}\left(\sum\limits_{c=0}^{c_{l-1}-1}\tau_c\right)\left(\sum\limits_{c=0}^{b-1}\tau_c\right)^{l-1}
+\ldots+\prod_{k=1}^l\tau_{c_k}\cdot\left(\sum\limits_{c=0}^{c_0-1}\tau_c\right)\\
&=\sum\limits_{i=0}^l\prod\limits_{k=i+1}^l\tau_{c_k}\cdot d(c_i) \cdot d(b)^i.
\end{array}
\end{equation}

Now we can make a distinction based on the absolute value of $d(b)=\sum_{c=0}^{b-1}\tau_c$. 

If $|\sum_{c=0}^{b-1}\tau_c|<1$, the sum $\sum_{n=0}^{N-1}\tau_n$ is bounded.

If $|\sum_{c=0}^{b-1}\tau_c|=1$, we get $\sum_{n=0}^{N-1}\tau_n=O(\log N)$.

If $|\sum_{c=0}^{b-1}\tau_c|>1$, we are going to prove that for any fixed value of $\log(\sum_{c=0}^{b-1}\tau_c)$, there exists a continuous and nowhere differentiable
function $F:[0,1]\to\C$ such that
\begin{equation}
\label{defF}
\sum_{n=0}^{N-1}\tau(n)=F\left(\log_b N\right)N^\frac{\scriptstyle\log(\sum_{c=0}^{b-1}\tau_c)}{\scriptstyle\log b} \text{ for all $N\in\N$.}
\end{equation}

Remark that the final expression given in \eqref{Spsi} can make sense for any (positive) real number.

\begin{definition} Given a real number $x>0$ and its base $b$ expansion of the form $\overline{c_lc_{l-1}\ldots c_0.c_{-1}c_{-2}\ldots}$ , define $\psi(x)=\psi_{\tau,b}(x)$ as
\begin{equation}
\label{defPsi}
\psi(x)=\sum_{i=-\infty}^l\prod_{k=i+1}^l\tau_{c_k}\cdot d(c_i) \cdot d(b)^i. 
\end{equation}
\end{definition}

This definition needs to be justified.

\begin{prop}
Suppose that $x$ is a number of the form $x=b^{-m}X$ with $m,X\in\N$. Then the expression in \eqref{defPsi} takes the same value for both expansions of $x$ in base $b$.
\end{prop}
\begin{proof}
Suppose that the expansion of $x$ of finite length is $\overline{c_lc_{l-1}\ldots c_0.c_{-1}c_{-2}\ldots c_{-m}}$, then the one of infinite \\[2pt]length is
$\overline{c_lc_{l-1}\ldots c_0.c_{-1}c_{-2}\ldots (c_{-m}-1)(b-1)(b-1)\ldots} $ and the proposition is equivalent to the identity
$$\begin{array}{rl}
\sum\limits_{i=-m}^{l} \prod_{k>i}\tau_{c_k}\cdot d(c_i) d(b)^i=&\sum\limits_{i=-m+1}^l \prod_{k>i}\tau_{c_k}\cdot d(c_i) d(b)^i 
+ \prod\limits_{k>-m}\tau_{c_k}\cdot d(c_{-m}-1) d(b)^{-m}\\
&+ d(b-1)\left(\sum\limits_{i=-\infty}^{-m-1} \prod\limits_{k>i}\tau_{c_k}\cdot d(b)^i \right).
\end{array}
$$
After cancelling the terms corresponding to $i>-m$ from the sums, we get:
$$\prod_{k>-m}\tau_{c_k}\cdot d(c_{-m}) d(b)^{-m}= \prod\limits_{k>-m}\tau_{c_k}\cdot d(c_{-m}-1) d(b)^{-m}
+ d(b-1)\left(\sum\limits_{i=-\infty}^{-m-1} \prod\limits_{k>i}\tau_{c_k}\cdot d(b)^i \right).$$ 
Subtracting the first term of the right-hand side from the left-hand side gives
$$\prod_{k\geqslant -m}\tau_{c_k}\cdot d(c_{-m}) d(b)^{-m}= d(b-1)\left(\sum\limits_{i=-\infty}^{-m-1} \prod\limits_{k>i}\tau_{c_k}\cdot d(b)^i \right).$$
After  cancellation in products we get the equivalent identity
$$d(b)^{-m}= d(b-1)\left(\sum\limits_{i=-\infty}^{-m-1} \prod\limits_{k=i+1}^{-m-1}\tau(b-1)\cdot d(b)^i \right)$$
which can be reduced by changing  indexes to
$$1=d(b-1)\sum\limits_{i=-\infty}^{-1}\tau(b-1)^{-i-1}d(b)^i. $$
The right-hand side of this expression is equal to
$$d(b-1)\sum_{i=1}^{+\infty}\tau(b-1)^{i-1}d(b)^{-i}=\frac{d(b-1)}{d(b)}\sum_{i=1}^{+\infty}\left(\frac{\tau(b-1)}{d(b)}\right)^{i-1}
=\left(1-\frac{\tau(b-1)}{d(b)}\right)\sum_{i=1}^{+\infty}\left(\frac{\tau(b-1)}{d(b)}\right)^{i-1}=1. $$

\end{proof}

By the same method we can get the two following propositions.

\begin{prop}
$\psi$ is continuous.
\end{prop}
\begin{proof}
Consider a sequence of real numbers $x_1,x_2,\ldots,x_n,\ldots$ converging to a real $x>0$. Suppose that $x_n>x$ for all $n$ or that $x_n<x$ for all $n$. We are going to choose a base $b$ expansion of $x$ accordingly: if $x_n>x$, take $\overline{c_lc_{l-1}\ldots c_0.c_{-1}c_{-2}\ldots}$ to be the expansion of $x$ not ending with $b-1$'s, and if $x_n<x$ take $\overline{c_lc_{l-1}\ldots c_0.c_{-1}c_{-2}\ldots}$ to be the one not ending with zeroes.

In both cases one can find for any $m>0$ a rank $N$ such that for any $n>N$, $x$ and $x_n$  have $m$ identical digits after the radix point. Therefore,
$$
\begin{array}{cl}
|\psi(x)-\psi(x_n)|&=\left|\sum\limits_{i=-\infty}^{m-1}\prod_{k>i}\tau_{c_k}\cdot d(c_i)d(b)^i -
\sum\limits_{i=-\infty}^{m-1}\prod_{k>i}\tau(\bar{c}_k)\cdot d(\bar c_i)d(b)^i\right|\\[5mm]
&\leqslant 2\max_{c\in\{0,\dots,b-1\}} \sum_{i<-m}|d(b)|^i \xrightarrow[m\to\infty]{}0
\end{array}
$$
where $\bar c_i$ denote the digits of $x_n$. Therefore, the sequence $\psi(x_n)$ converges to $\psi(x)$.
\end{proof}

\begin{prop}
$\psi$ is nowhere differentiable.
\end{prop}
\begin{proof}
We are going to to use the following characterizaion of the derivative: if $f$ is a real function derivable at $x$ and $f'(x)=c$ then for every $\epsilon>0$ there exists $\delta>0$ such that for every $x_1,x_2\in\R$ verifying $x-\delta<x_1\leqslant x\leqslant x_2<x+\delta$ and $x_1<x_2$ we get
$$c-\epsilon<\frac{f(x_2)-f(x_1)}{x_2-x_1}<c+\epsilon. $$

Let $x$ be a positive real number represented in base $b$ by the sequence of digits $\overline{c_lc_{l-1}\ldots c_0.c_{-1}c_{-2}\ldots}$ that does not end by $b-1$'s. Let $J_n$ be the strictly increasing sequence of indexes such that $c_{-J_n}<b-1$. For each $n$ denote
$$x_n=\overline{c_lc_{l-1}\ldots c_0.c_{-1}c_{-2}\ldots c_{J_n}} $$
and
$$y_n=\overline{c_lc_{l-1}\ldots c_0.c_{-1}c_{-2}\ldots (c_{J_n}+1)}. $$
Then $y_n>x\geqslant x_n$ and $\lim_{n\to\infty}(y_n-x_n)=0$ and
$$\psi(y_n)-\psi(x_n)=\prod_{k\geqslant J_n}\tau_{c_k}d(b)^{-J_n}. $$
Therefore
$$\left|\frac{\psi(y_n)-\psi(x_n)}{y_n-x_n}\right|=\left(\frac{d(b)}{b}\right)^{-J_n}\xrightarrow[n\to\infty]{} +\infty. $$
This proves that the sequence $\frac{\psi(y_n)-\psi(x_n)}{y_n-x_n}$ cannot converge to any finite complex number, and, as consequence, the function $\psi$ cannot be differentiable at $x$.
\end{proof}

By its definition, $\psi(bx)=d(b)\psi(x)$. Now we can define the function $F$ satisfying the formula \eqref{defF}. For any $x>0$ denote
\begin{equation}
\label{defF2}
F\left(\log_b x\right)=\psi(x) x^{- \frac{\scriptstyle\log d(b)}{\scriptstyle\log b}}. 
\end{equation}
This is a continuous function of a real argument. In the formula \eqref{defF2}, $\log d(b)$ can be taken to be any fixed value, and the other logarithms are those of positive real arguments, so they are supposed to be real. $F$ is also periodic of period $1$ which finishes the proof of \eqref{defF}. Indeed, if $y=\log_b x$, then
$$F(y+1)=F\left(\log_b(bx)\right)=d(b)\psi(x)\cdot b^{- \frac{\scriptstyle\log d(b)}{\scriptstyle\log b}}x^{- \frac{\scriptstyle\log d(b)}{\scriptstyle\log b}}=\psi(x)x^{- \frac{\scriptstyle\log d(b)}{\scriptstyle\log b}}=F(y). $$


The problem of Newman phenomenon can be challenged using this result. Suppose that $t_n$ is a $b$-multiplicative sequence satifying the above hypotheses, $p$ is a prime number coprime to $b$ and $\zeta$ is a $p$-th root of unity. Then the $p$-rarefied sum of $t_n$ can be written as:
$$\sum_{n<N}\ind_{p|n} t_n
= \sum_{n<N}\frac1p\left(1+\zeta^n+\zeta^{2n}+\ldots+\zeta^{(p-1)n}\right)t_n
=\frac1p\left(\sum_{n<N}t_n+\sum_{n<N}\sum_{j\in\F^\times}\zeta^{jn}t_n\right). $$
If we define $s$ to be the order of $b$ in the group $\F^\times$, we obtain that $\zeta^{jn}t_n$ are $b^s$-multiplicative sequences. Applied to the particular case of the Thue-Morse sequence as $t_n$, our method is the direct generalization of the method used in \cite{Coquet} and \cite{Grabner}.

The following chapters will be concentrated on one example of study of the constant $d(b^s)$ associated to sequences of the form $\tau_n=\zeta^{jn}t_n$ where $t_n$ is a precise sequence of \ $+1$s and $-1$s. We hope that the tools developed for this example will be useful for similar  questions. But before going to particular cases, let's add a general remark about the rarefied $+1/-1$ sequences.

Consider, as before, a $b$-multiplicative sequence $t_n$ satisying the conditions stated at the beginning of this chapter and composed only of numbers $1$ and $-1$, let $p$ be a prime number coprime to $b$, and let $s,\zeta$ and the sequence $\tau$ be as above, let $<b>$ be the subgroup of $\F^\times$ generated by $b$. Then the constant $d(b^s)$ associated to $\tau$ is
$$d(b^s)=\sum\limits_{n=0}^{b^s-1}t_n\zeta^n $$ 
and it can alo be written, by the $b$-multiplicativity of $t_n$ as
$$d(b^s)=\prod_{k=0}^{s-1}\left(\sum_{c=0}^{b-1}t_c\zeta^{b^kc}\right)=\prod_{j\in<b>}\left(\sum_{c=0}^{b-1}t_c\zeta^{jc}\right). $$

It follows from this formula that $d(b^s)$ is an element of a number field of degree at most $\frac{p-1}{s}$. Indeed, in the group $\F^\times$ there are $\frac{p-1}{s}$ classes modulo $<b>$ and to each class $[j]$ we can associate the number
$$\xi^{[j]}=\sum\limits_{\textstyle n=0}^{\textstyle b^s-1}t_n\zeta^{jn}.  $$
All the symetric polynomials in the numers $\xi^{[j]}$ are integers which proves our claim. 

Given this result, it is natural to study the numbers 
$$\xi^{[i]}=\prod_{j\in i\F^\times k} (\sum_{c=0}^{b-1}t_c\zeta^{jc})$$
(regardless on whether $s=\frac{p-1}{k}$) instead of numbers $d(b^s)$; this will be the perspective of the following sections.

Let us suggest a notation for the $b$-multiplicative sequences composed of $1$s and $-1$s by the sequence of '$+$' and '$-$' signs corresponding to the $b$ initial terms of the sequence. For instance, with this notation, the Thue-Morse sequence is the sequence $"+ -"$.

\section*{3. Poker combinatorics.}
\refstepcounter{section}
This section  can be seen as a box of useful tools. Its main result is the following:

\begin{lemma}
\label{delta1d}
Let $p$ be a prime number and $0\leqslant n<p$ an integer. Denote $A_0(n,p)$ the number of subsets of $\Fx$ of $n$ elements that sum up to $0$ (modulo $p$) and $A_1(n,p)$ the number of those subsets that sup up to $1$. Then
$$A_0(n,p)-A_1(n,p)=(-1)^n. $$ 
\end{lemma}

Let's begin the proof with an obvious observation: if we define in analogous way the numbers $A_2(n,p)$, $A_3(n,p)$, \dots, $A_{p-1}(n,p)$ , they are all equal to $A_1(n,p)$ since multiplying a set that sums to $1$ by a constant residue $c\in\Fx$ gives a set that sum up to $c$, and this correspondence is a bijection.

Let us deal with a simpler version of the lemma that includes order of elements and repetitions, that is, counts the following objects:

\begin{definition}
Denote $E^{k_1,k_2,\ldots,k_n}_x(n,p)$ (where $x\in\F$ and $k_1,k_2,\ldots,k_n\in\Fx$) the number of sequences $(x_1,x_2,\ldots,x_n)$ of elements of $\Fx$ such that
$$\sum_{i=1}^n k_i x_i=x. $$
We will also denote $E^{1,1,\ldots,1}_x(n,p)$ by $E_x(n,p)$.
\end{definition}
Then we get the following

\begin{prop}
\label{E}
If $n$ is even,
$$E^{k_1,k_2,\ldots,k_n}_0(n,p)=\frac{(p-1)^n+p-1}{p}\quad \text{\rm  and }\quad E^{k_1,k_2,\ldots,k_n}_1(n,p)=\frac{(p-1)^n-1}{p};$$
\hspace{0.5cm}if $n$ is odd,
$$E^{k_1,k_2,\ldots,k_n}_0(n,p)=\frac{(p-1)^n-p+1}{p}\quad \text{\rm  and }\quad E^{k_1,k_2,\ldots,k_n}_1(n,p)=\frac{(p-1)^n+1}{p}.$$
In both cases,
$$E^{k_1,k_2,\ldots,k_n}_0(n,p)-E^{k_1,k_2,\ldots,k_n}_1(n,p)=(-1)^n. $$ 

\end{prop}
\begin{proof}
\it By induction on n.\rm\  For $n=0$ or $n=1$ the result is trivial. For bigger $n$ we always get:
$$E_0^{k_1,k_2,\ldots,k_n}(n,p)=(n-1)E_1^{k_1,k_2,\ldots,k_{n-1}}(n-1,p) $$
and
$$E_1^{k_1,k_2,\ldots,k_n}(n,p)=E_0^{k_1,k_2,\ldots,k_{n-1}}\textit{}(n-1)+(p-2)E_1^{k_1,k_2,\ldots,k_{n-1}}(n-1,p) $$
since the sequences of length $n$ summing up (with coefficients) to $x$  are exactly expansions of sequences of length $n-1$ summing up (again, with coefficients) to another residue than $x$ and the correspondence is a bijection. Injecting formulas for $n-1$ concludes the induction.
\end{proof}

Now we are going to prove Lemma \ref{delta1d} for small $n$. If $n=0$ or $n=1$, the lemma is clear. For $n=2$, there is one more sequence $(x,y)\in\Fx{^2}$ that sums up to $0$, but that counts the sequences of the form $(x,x)$ which should be thrown out. Since $p$ is prime, these sequences contribute once for every non-zero residue modulo $p$, and throwing them out increases the zero's advantage to $2$. Now, we have to identify each couple $(x,y)$ and $(y,x)$ to be the same, so we get the difference $1$ back, establishing the lemma $1$ for $n=2$.

For $n=3$, counting all the sequences $(x,y,z)\in\Fx$ gives a difference $E_0-E_1=-1 $. The sequences $(x,x,z)$ contribute one time more often to the sum equal to $0$, so throwing them out adds $-1$ to the total difference. The same thing applies to sequences of the form $(x,y,y)$ and $(x,y,x)$. Done that, we get an intermediate  difference of $-4$, but the triples of the form $(x,x,x)$ have been thrown out $3$ times, which is equivalent to saying they count $-2$ times. So they should be "reinjected" with the coefficient $2$. As $p$ is prime and  bigger than $3$, the redundant triples contribute once for each non-zero residue; therefore we accumulate the difference of $-4-2=-6$. We have then to idetify permutations, that is to divide score by $6$ which gives the final result $-1$.

Now, here is the explicit calculation for the case $n=4$:

$$\begin{array}{ll}
1 & \text{\rm  (corresponds to $E_0(4,p)-E_1(4,p)$)}\\
+6 & 	\text{\rm  (for throwing out \begin{tabular}{ccc}  $(x,x,y,z)$,& $(x,y,x,z)$,&$(x,y,z,x)$,\\ $(x,y,y,z)$,&$(x,y,z,y)$,&$(x,y,z,z)$ \end{tabular})} \\[2mm]
+2\times 4 & \text{\rm  (for reinjecting $(x,x,x,y),(x,x,y,x), (x,y,x,x)$ and $(x,y,y,y)$)}\\[2mm]
+1\times 3 & \text{\rm  (for reinjecting $(x,x,y,y),(x,y,x,y)$ and $(x,y,y,x)$)}\\[2mm]
+6\times 1 & \text{\rm  (for throwing out $(x,x,x,x)$)}\\[2mm]
=24 &
\end {array}
$$
which is $4!$, therefore the lemma $1$ is proved for $n=4$.

For a general $n$ we can calculate the difference between the number of sequences that give $0$ and the number of those giving $1$ by assigning to all sequences in $\Fx{^n}$ an intermediate coefficient equal to one, then by reducing it by one for each couple of equal terms, then increasing by $2$ for each triple of equal terms, and so on, proceding by successive adjustments of coefficients, each step corresonding to a "poker combination" of $n$ cards. If after adding the contributions of all the steps and the initial $(-1)^n$, we get $(-1)^nn!$, Lemma \ref{delta1d} is valid for $n$ independently from $p$ provided that $p>n$ is prime.

In  what follows we are going to formalize the concept of poker combination using the notions exposed in \cite{Rota}. Call a \it partition \rm of the set $\{1,2,\dots,n\}$ a choise of pairwize disjoint nonempty subsets $B_1,B_2,\dots, B_c$ of $\{1,2,\dots,n\}$ of nonincreasing sizes $|B_i|$, and such that $B_1\cup B_2 \cup\dots\cup B_c=\{1,2,\dots,n\} $. The set $\Pi_n$ of all partitions of  $\{1,2,\dots,n\}$ is partially ordered by reverse refinement: for each two partitions $\tau$ and $\pi$, we say that $\tau\geqslant\pi$ if each block of $\pi$ is included in a block of $\tau$. We define the M\"obius function $\mu(\hat 0,x)$ on $\Pi_n$ (the definition and notation are due to \cite{Rota}) recursively by:\\
\quad if $x=\{\{1\},\{2\},\cdots,\{n\}\}=\hat 0$, then $\mu(\hat 0,x)=1$;\\
\quad if $x$ is bigger than $\hat 0$, then 
$$\mu(\hat 0,x)=-\sum_
{\begin{array}{c}
y\in\Pi_n\\
y<x
\end{array}
}\mu(\hat 0,y). $$

Figure $1$  illustrates  the structure of $\Pi_n$ in the case $n=7$. It brings together all the partitions of $\{1,2,\dots,7\}$ of the same type that is with the same sequence $(B_1,B_2,\cdots,B_n)$ of sizes of blocks and indicates in the square brackets the values of the M\"obius function associated to each type: 

\begin{figure}[h]
\begin{center}
\begin{picture}(290,310)
\put(150,0){$\hat 0_{<+1>}$}
\put(140,50){$\rm Pair_{<-1>}$}
\put(70,100){$2\rm p_{<+1>}$}
\put(210,100){$3\text{ of a kind}_{<+2>}$}
\put(0,150){$3\rm p_{<-1>}$}
\put(75,150){$4\text{ of a kind}_{<-6>}$}
\put(175,150){$\text{Full House}_{<-2>}$}
\put(50,200){$4\pl 2_{<+6>}$}
\put(140,200){$5\text{ of a kind}_{<+24>}$}
\put(245,200){$3\pl 3_{<+4>}$}
\put(0,250){$3\pl 2\pl 2_{<+2>}$}
\put(70,300){$4\pl 3_{<-12>}$}
\put(140,300){$6\text{ of a kind}_{<-120>}$}
\put(245,275){$5\pl 2_{<-24>}$}
\put(145,350){$7\text{ of a kind}_{<+720>}$}

\put(152,11){\line(0,1){38}}

\put(150,58){\line(-5,3){65}}
\put(180,58){\line(5,3){68}}

\put(72,108){\line(-5,3){66}}
\put(75,108){\line(1,3){12}}
\put(80,108){\line(5,2){100}}

\put(218,108){\line(-3,1){118}}
\put(240,108){\line(-2,3){25}}

\put(10,158){\line(0,1){90}}
\put(15,158){\line(1,1){40}}

\put(90,158){\line(-3,4){30}}
\put(103,158){\line(3,2){60}}

\put(185,158){\line(-3,1){120}}
\put(200,158){\line(-2,1){172}}
\put(205,158){\line(-1,1){40}}
\put(210,158){\line(3,2){60}}

\put(55,208){\line(1,5){17.5}}
\put(60,208){\line(1,1){90}}
\put(80,206){\line(5,2){170}}

\put(160,208){\line(0,1){90}}
\put(187,208){\line(1,1){65}}

\put(250,208){\line(-2,1){175}}
\put(255,208){\line(-1,1){90}}

\put(45,258){\line(2,3){26}}

\put(75,308){\line(2,1){75}}

\put(160,308){\line(0,1){40}}

\put(265,283){\line(-3,2){98}}

\end{picture}
\end{center}
\caption{The partitions of $7$ items.}
\label{pokerDragon}
\end{figure}
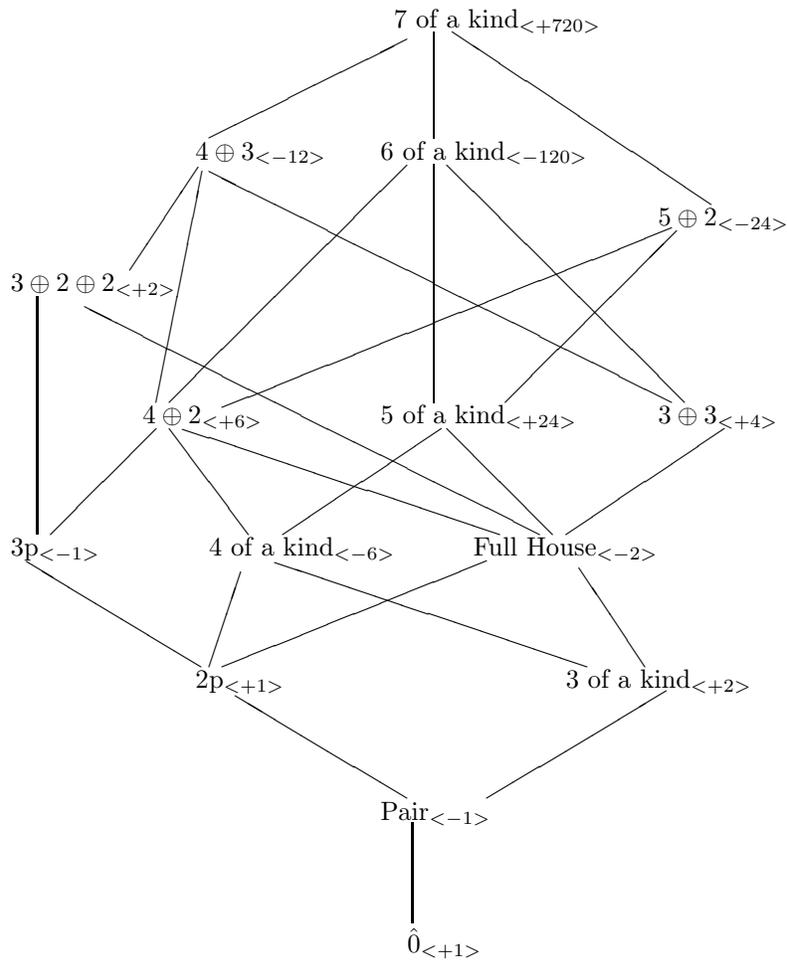

This figure uses the following notations (borrowed from poker) for types of partitions:\\
$\bullet$ Pair: the type $(2,1,1,\ldots,1)$.\\
$\bullet$ $2$p, $3$p : two and three pairs respectively i.e. the types $(2,2,1,1,1)$ and $(2,2,2,1)$.\\
$\bullet$ $n$ of a kind: the type $(n,1,1,\ldots,1)$.\\
$\bullet$ Full House:  $(3,2,1,1)$.\\
$\bullet$ $n_1\pl n_2\pl\ldots\pl n_k$ : the type $(n_1,n_2,\ldots,n_k,1,1\ldots,1)$.

By the Corollary to the Proposition $3$ section $7$ of \cite{Rota0} and the first Theorem from the section $5.2.1$ of \cite{Rota}, if $x$ is a subdivision of type $(\lambda_1,\lambda_2,\cdots,\lambda_n)$, then
$$\mu(\hat 0,x)=\prod_{i=1}^{n}(-1)^{\lambda_i-1}(\lambda_i-1)! $$
but we won't need this formula directly.

We are also going to use the following definition: let $s=(x_1,x_2,\cdots,x_n)$ be a sequence of $n$ nonzero residues modulo $p$ regarded as a function
$$s:\{1,2,\cdots,n\}\to\Fx. $$ 
Then its \it coimage \rm is the partition of $\{1,2,\cdots,n\}$ whose blocks are the nonempty preimages of elements of $\Fx$. Now we can prove the following proposition that puts together all the previous study.
\begin{prop}
\label{deltaInv}
The difference
$$A_0(n,p)-A_1(n,p) $$
does not depend on $p$ provided that $p$ is a prime number bigger than $n$.
\end{prop}
\begin{proof}
We are going to describe an algorithm that computes this difference (which is the one applied earlier for small values of the argument). For each subdivison $x$ of the set $\{1,2,\cdots,n\}$, denote by $r_0(x,p)$ the number of sequences $(x_1,x_2,\ldots,x_n)$ of elements of $\Fx$  of coimage $x$ that sum up to $0$, denote by $r_1(x,p)$ the number of those sequences of coimage $x$ that sum up to $1$ and denote $r(x,p)=r_0(x,p)-r_1(x,p)$. Then,
$$n!(A_0(n,p)-A_1(n,p))=r(\hat 0,p). $$

Denote, for each subdivision $y$ of $\{1,2,\cdots,n\}$, 
$$s(y,p)=\sum_{x\geqslant y} r(x,p). $$

Then, by Proposition \ref{E},
$$s(y,p)=(-1)^{c(y)} $$
where $c(y)$ is the number of blocks in the subdivision $y$. By the M\"obius invesion formula (see \cite{Rota}),

$$r(\hat 0,p)=\sum_{y\in\Pi_n} \mu(\hat 0,y)s(y,p)=\sum_{y\in\Pi_n}(-1)^{c(y)}\mu(\hat 0,y) . $$

If we compute this sum, we get the value of $A_0(n,p)-A_1(n,p)$ independently from $p$.
\end{proof}

The last move  can seem artificial\footnote{A purely compbinatorial and more general proof exists: see the Chapter $3$ of \cite{EnumCombin}, formula $(31)$ } but it is sufficient to complete the proof of Lemma \ref{delta1d}. Remark that $A_0(n,p)=A_0(n,p-1-n)$ since saying that the sum of some subset of $\Fx$ is $0$ is equivalent to saying that the sum of its complement is $0$. For the same kind of reason, $A_1(n,p)=A_{-1}(n,p-1-n)=A_1(n,p-1-n)$. 

Now we can prove Lemma \ref{delta1d} by induction on $n$. It has already been proved for small values of $n$. If $n>4$, by Bertrand's postulate, there is a prime number $p'$ such that $n<p'<2n$. Replace $p$ by $p'$ (by the proposition \ref{deltaInv} this leads to an equivalent statement) then $n$ by $p'-1-n$ (using the above remark). As $p'-1-n<n$, the step of induction is done.

The two following consequences of the lemma \ref{delta1d} have, in fact,  a much simpler proof (see, for example \cite{NumTh}). 

\begin{corollary}
Let $p$ be prime. Then, in the polynomial ring $\Z[T]/(T^p-1)$, 
$$\prod_{j\in\Fx}(1-T^j)=(p-1)-\sum_{i\in\Fx}T^i. $$
\end{corollary}
\begin{proof}
By expanding the product we get:
$$\prod_{j\in\Fx}(1-T^j)=C_0+C_1T+C_2T^2+\ldots+C_{p-1}T^{p-1} $$
where
$$C_i=\sum_{n=0}^{p-1}(-1)^nA_i(n,p).$$
By the properties of the numbers $A_i(n,p)$ , $C_1=C_2=\ldots=C_{p-1}$ and (lemma \ref{delta1d}) $C_0-C_1=p$. As the sum of all the $C_i$ is $0$, we get necessarily
$$C_0=p-1\text{ and }C_1=-1. $$
\end{proof}

\begin{corollary}
Let $p$ be a prime number and $\zeta$ be a primitive $p$-th root of unity. Then,
$$\prod_{j\in\Fx}(1-\zeta^j)=p. $$
\end{corollary}
\begin{proof}
This is trivial from the previous result.
\end{proof}


\section*{4. Some lemmas about Lucas numbers.}
\refstepcounter{section}
The Lucas numbers have the same formula as Fibonacci numbers but different starting terms. Their definition is as follows:
\begin{definition}
The sequence $L_n$ of Lucas numbers is defined by:
\begin{align}
L_0&=2,\notag\\
L_1&=1,\notag\\
\intertext{and for every $n\in\N$}
L_{n+2}&=L_{n}+L_{n+1}.\label{defLucas}
\end{align}
This is the sequence $A000032$ in OEIS.
\end{definition}

To make clear the problem of offset, let us define the version of Fibonacci sequence that will be used.
\begin{definition}
The sequence of Fibonacci numbers is defined by
\begin{align}
F_0&=0,\notag\\
F_1&=1,\notag\\
\intertext{and for every $n\in\N$}
F_{n+2}&=F_{n}+F_{n+1}.\label{defFibo}
\end{align}
\end{definition}

This chapter collects some properies of Lucas numbers which are analogous to properties of Fibonacci numbers but are difficult to find in literature. The first one draws the relation between both sequences.
\begin{prop}
$$\forall n\geqslant 1\quad L_{n}=F_{n-1}+F_{n+1}. $$
\end{prop} 
\begin{proof}
For $n$ equal $1$ and $2$ this formula is checked directly. For bigger values of $n$, it follows from indutive definitions of Lucas and Fibonacci numbers.
\end{proof}

Let us recall the combinatorial definition of Fibonacci numbers: if $n\geqslant2$, $F_n$ is the number of ways to put disjoint "dominos" (i.e., subsets of the form $\{k,k+1\}$) on the integer interval $[1,n-1]$. The Lucas numbers have a similar characterization:
\begin{prop}
If $n\geqslant3$, $L_n$ is the number of ways to put disjoint "dominos" (i.e., subsets of the form $\{k,k+1\}$) on the circle $\Z/n\Z$.
\end{prop}
\begin{proof}
Distinguish three types of placements of dominos following what happens to a fixed element of $\Z/n\Z$, say to $0$. 

If  in no domino is located at $0$, these placements are equivalent to the placements of dominos on the interval $[1,n-1]$ and we count $F_n$ of these.

If there is a domino $'0 \Minus 1'$, then the remaining have to be placed on the interbval $[2,n-1]$ and we count $F_{n-1}$ ways of doing it.

If there is a domino $'(p-1) \Minus 0'$, the count is the same with  the interval $[1,p-2]$.

The sum of these  three numbers is $F_n+F_{n-1}+F_{n-1}=F_{n+1}+F_{n-1}=L_n$. 
\end{proof}

The last property interesting in the perspective of this article is purely of number-theoretical kind.
\begin{theorem}
\label{lucas14}
Let $n\geqslant3$ be an odd integer. Then all prime factors of $L_n$ are either $2$ or congruent to $1$ or $-1$ modulo $5$. 
\end{theorem}
Its proof follows the one that characterizes the sums of two squares. let's begin with the following.
\begin{prop}
For each $n\geqslant 1$, $L_{n+1}L_{n-1}=L_n^2+(-1)^{n+1}5 $.
\end{prop}
\begin{proof}
\it By induction on n.\rm\  For $n=1$, the statement is clear. For bigger $n$, it is equivalent to
$$
\begin{array}{cl}
L_{n-1}^2+L_nL_{n-1}=L_n^2+(-1)^{n+1}5 &\text{(by applying \eqref{defLucas} to $L_{n+1}$),}\\[2pt]
L_{n-1}^2=L_n(L_n-L_{n-1})+(-1)^{n+1}5, & \\[2pt]
L_{n-1}^2+(-1)^{n}5=L_nL_{n-2} & \text{(by applying it once more to $L_n-L_{n-1}$),}
\end{array}
$$
which is the proposition \thetheorem\ for $n-1$.
\end{proof}

\it Proof of Theorem \ref{lucas14}.\rm\ 
For Lucas numbers with odd index, the proposition  \thetheorem\ takes the form:
$$L_nL_{n-2}=L_{n-1}^2-5 \text{\rm\  (where $n$ is odd).}$$

If $p$ is an odd prime number congruent to $2$ or $3$ modulo $5$, by the quadratic reciprocity law the congruence
$$L^2-5\equiv 0 \text{\rm\  mod }p $$
has no integer solutions, which implies that $p$ cannot divide $L_n$.\hfill$\Box$

\section*{5. The particular case of the Thue-Morse sequence.}
\refstepcounter{section}
This chapter and the next one will  deal with the problem of study of the numbers 
\begin{equation}
\label{xiA}
\xi^{[a]}=\prod_{j\in a\Gamma}(\sum_{c=0}^{b-1}t_c\zeta^j)
\end{equation}
 where $\zeta$ is a primitive $p$-th root of unity ($p$ prime), $\Gamma$ is a subgroup of $\Fx$ and $t_c\in\{1,-1\}$. The case corresponding to the rarefied Thue-Morse sequence (that is, $b=2$, $t_0=1$ and $t_1=-1$) is the one that has been studied rather extensively in the article \cite{GKS}; in this chapter we are going to recall their results and prove refined versions of some of them.

The first result is: for any sequence $t_n$ if the subgroup $\Gamma$  is of even order, all the numbers $\xi^{[a]}$ are real positive. Indeed, in this case $-1\in\Gamma$ therefore the product \eqref{xiA} is composed of pairs of complex-conjugate terms. If  $\Gamma$ is of odd order and the sequence $t_n$ is the Thue-Morse sequence, then 
$$\prod_{j\in a\Gamma} (1-\zeta^j)=\prod_{j\in\frac12a\Gamma}(\zeta^{-j}-\zeta^j)=\prod_{j\in\frac12a\Gamma}\left(2i\sin\frac{2\pi i}p\right) $$
where $\frac12 $ corresponds to the inverse of $2$ in $\F$. Therefore, in this case the numbers $\xi^{[a]}$ are pure imaginary (this is a stronger version of the Proposition $3.3$  of\cite{GKS}).

Therefore, the numbers $\xi^{[1]}$ and $\xi^{[i]}$ are explicitly described by their product (equal to $p$) and their sum. Some values of the latter are given in the appendix A.

\section*{6. The case of the sequence "$++-$".}
\refstepcounter{section}
In this chapter  we are going to answer one question about the sequence $"++-"$ that is, the $3$-multiplicative sequence defined by
$$t_n=(-1)^{\text{\rm the number of '$2$' digits in the ternary expansion of $n$}}. $$
Its initial terms are:
$$11\uo\ 11\uo\ \uo\uo1\ 11\uo\ \ldots\quad\text{(where $\uo$ stands for $-1$);} $$
this can be seen as the "simplest" particular case of our subject after the Thue-Morse sequence. The main result of this section is the following one:
\begin{theorem}
\label{Lucas}
Let $p$ be an odd prime number and $\zeta$ a primitive $p$-th root of unity. Then,
$$ N_{\Q(\zeta)/\Q}(1+\zeta-\zeta^2)=L_p,$$
$L_p$ being the $p$-th Lucas number.
\end{theorem}

The case $p=3$ is irrelevant to the study of the rarefied $"++-"$ sequences since the $3$-rarefied sequence $t_{3n}$ is identical to $t_n$. The first terms of the sequence $L_p$ where $p\geqslant5$ is prime, are given in the Appendix B.

\it Proof of the theorem \ref{Lucas}.\rm \\
We are going to adapt the methods from the Section $3$ to prove this result. Begin with some new notations.
\begin{definition}
If $f$ is a linear application from $\F^{p-1}$ to $\F$ of the form
\begin{equation}
\label{flin}
f(x_1,x_2,\ldots,x_{p-1})=f_1x_1+f_2x_2+\ldots+f_{p-1}x_{p-1} 
\end{equation}
where the coefficients $f_1,f_2,\ldots,f_{p-1}$ are $0$,$1$ or $2$, and $i$ is an element of $\F$, denote by $B_i(f,p)$ the number of sequences $(x_1,x_2,\ldots,x_{p-1})$ that are permutations of $\Fx$ and such that $f(x_1,x_2,\ldots,x_{p-1})=i$, and by $A_i(f,p)$ the same thing but permutations between places corresponding to the same coefficient $f_k$ inside the sequence $(x_1,x_2,\ldots,x_{p-1})$ (that is, permutations leading to the same expression in \eqref{flin}) are considered to be the same. 
\end{definition}

We are going to denote by $n_0(f)$ the number of coefficients of $f$ equal to $0$, by $n_1(f)$ the number of those equal to $1$ and $n_2(f)$ the number of those equal to $2$.

For example, if the coefficients of $f$ are only $0$s and $1$s 
, then
$$A_i(f,p)=A_i(n_1(f),p) \text{\rm\quad  and\quad } B_i(f,p)=n_0(f)!n_1(f)!A_i(n_1(f),p).$$

As before, the numbers $B_i(f,p)$ are the same for each $i\in\Fx$. If we calculate the norm of $(1+\zeta-\zeta^2)$ by expanding the product, we get the following: in $\Z[T]/(T^p-1)$,
$$\prod_{j\in\Fx}(1+T^j-T^{2j})=C_0+C_1T+C_2T^2+\ldots+C_{p-1}T^{p-1}$$ 
where
\begin{equation}
\label{xiPPM}
C_i=\sum_{f}(-1)^{n_2(f)}A_i(f,p) 
\end{equation}
the last sum running on all linear combinations satisfying the conditions in definition \arabic{definition}. And,
\begin{equation}
\label{c0c1}
N(1+\zeta-\zeta^2)=C_0-C_1. 
\end{equation}

The differences $A_0(f,p)-A_1(f,p)$ can be calculated in this new setting. First, in all cases, $B_i(f,p)=n_0(f)!n_1(f)!n_2(f)!A_i(f,p)$. If the sum of all coefficients of $f$ is smaller than $p$, then the algorithm of successive  readjustments of coefficients, described in the proposition $5$, together with the lemma $1$ proves that
\begin{align}
B_0(f,p)-B_1(f,p)&=(-1)^{n_1(f)+n_2(f)} (n_1(f)+n_2(f))!n_0(f)! \label{SSB} \\
\intertext{which implies}
A_0(f,p)-A_1(f,p)&=(-1)^{n_1(f)+n_2(f)} {n_1(f)+n_2(f)\choose n_1(f)}.\label{SmallSum} &&
\end{align}

If the sum of coefficients of $f$ is greater than (or equal to) $p$, formulae \eqref{SSB} and \eqref{SmallSum} no longer work because we can get $p$ on intermediate stages when adding coefficients of $f$. To help to this, we can replace $f$ by the linear combination $\mathbf{2}-f$ whose $k$-th coefficient is $(2-f_k)$. Both functions are related by the formula
$$(\mathbf{2}-f)(x_1,x_2,\ldots,x_{p-1})=-f(x_1,x_2,\ldots,x_{p-1}) $$
and the new linear combination's sum of coefficients is $2(p-1)-(2n_2(f)+n_1(f))\!<\!p$ if ${2n_2(f)+n_1(f)\geqslant p}$. Accordingly to this, the associated difference is
\begin{equation}
\label{BigSum}
A_0(f,p)-A_1(f,p)=A_0(\mathbf{2}-f,p)-A_1(\mathbf{2}-f,p)=(-1)^{n_0(f)+n_1(f)} {n_0(f)+n_1(f)\choose n_0(f)}. 
\end{equation}

If we put both formulas \eqref{SmallSum} and \eqref{BigSum} into \eqref{xiPPM} and \eqref{c0c1}, we get the following expression (the linear combinations in \eqref{xiPPM} are indexed first by $n_0(f)$ then by $n_2(f)$):
\begin{equation}
\label{binomials}
N(1+\zeta-\zeta^2)=\sum_{n_0=0}^{p-1}\left(\sum_{n_2=0}^{\min(p-1-n_0,n_0)}(-1)^{p-1-n_0-n_2}{p-1-n_0 \choose n_2} 
+\sum_{n_2=n_0+1}^{p-1-n_0}{p-1-n_2\choose n_0} \right).
\end{equation}

Now we are going to identify these numbers as numbers of ways to put dominos on the circle $\Z/p\Z$ with restrictions. The terms of the first sum that correspond to $n_0\!=\!\frac{p-1}{2},\ldots,p-2$ are zero, and the term corresponding to $n_0=p-1$ is one. Further, each term corresponding to $n_0=0,\ldots,\frac{p-3}{2}$ is equal to the number of putting $n_0+1$ dominos on the circle $\Z/p\Z$.

Indeed, call the first domino the first domino encountered when we walk on the circle in the direction $0,1,2,\ldots$  If no domino is located at $'(p-1)\Minus 0'$, suppose that the first domino is located at $'(n_2-n_0-1)\Minus (n_2-n_0)'$ leaving $p-1-n_2+n_0$ free sqares after it for the remaining $n_0$ dominos. The number of ways to put $n_0$ dominos on the line strip of length $p-1-n_2+n_0$ is exactly ${p-1-n_2+n_0-n_0\choose n_0}$ which is equal to a term of the last sum of \eqref{binomials}.
 
The second sum of \eqref{binomials} can be reduced to a more compact form by the following
\begin{prop}
For each $m,n\in\N$ such that $n\ge m$, 
$$\sum_{k=0}^m(-1)^{k+n}{n\choose k}=(-1)^{m+n}{n-1\choose m}. $$
\end{prop} 
\begin{proof}
\it By induction on m.\rm\  For $m=0$, the statement is trivial. To prove it for $m+1$, apply the induction hypothesis and get:
$$
\begin{array}{ll}
\displaystyle\sum_{k=0}^{m+1}(-1)^{k+n}{n\choose k}=&\displaystyle(-1)^{m+n}{n-1\choose m}+(-1)^{m+n+1}{n\choose m+1} \\
&\displaystyle=(-1)^{m+n+1}\left({n\choose m+1}-{n-1\choose m}\right) =(-1)^{m+n+1}{n-1\choose m-1}
\end{array}
$$
\end{proof}

By this proposition, if $n_0<\frac{p-1}{2}$, the second sum of \eqref{binomials} equals 
$$\sum_{n_2=0}^{\min(p-1-n_0,n_0)}(-1)^{n_0+n_2}{p-1-n_0 \choose n_2}={p-2-n_0 \choose n_0} $$
which is the first term of the third sum which, itself, equals the number of placements of $n_0+1$ dominos such that one domino is placed at $'0\Minus 1'$. We can identify this term to the number of placements of $n_0+1$ dominos with one domino placed at $'(p-1)\Minus 0'$ thus completing the description of all possible placements of dominos on the circle $\Z/p\Z$.\hfill$\Box$

\section*{7. Further questions.}
In the current state of things, we know much about the Thue-Morse sequence but  we only answered the first question about the sequence $"++-"$. It could be interesting, on one hand, to see whether the techniques developed here can be used to expand our knowledge about the rarefied Thue-Morse sequences, and therefore, about the class number of the quadratic fields, and, on the other hand, to prove for the sequence $"++-"$ analogous results to those we have already for the Thue-Morse sequence. Also, the sequence $"++-"$ has no reason to be unique in its kind: other multiplicative sequences can have the similar properties.

Furthermore, I think that Carthage must be destroyed\footnote{Cato the Elder, 234-149 BC}.

\newpage

\section*{Appendices.} 

\subsection*{A. Traces of $\xi^{[1]}$ corresponding to the subgroup of squares in $\Fx$ in the Thue-Morse case. }
The following table lists some values of $\Tr_{\Q(\sqrt{p})/\Q}N_{\Q(\zeta)/\Q(\sqrt{p})}(1-\zeta)$ where $p$ is a prime number congruent to $1$ modulo $4$ and $\zeta$ is a primitive $p$-th root of unity. These numbers are given in the form of product of numbers that are products of $2$ and $5$, and of pseudoprimes.

\begin{tabular}{|c|c|}
\hline
\(p\) & $\Tr N(1-\zeta)$ \\ 
\hline
$5$&$5$\\
$13$&$13$\\
$17$&$2\cdot17$\\
$29$&$29$\\
$37$&$2\cdot37$\\
$41$&$10\t41 $\\
$53$&$53 $\\
$61$&$5\t61 $\\
$73$&$250\t73 $\\
$89$&$2\t53\t89 $\\
$97$&$2\t97\t569 $\\
$101$&$2\t101 $\\
$109$&$25\t109 $\\
$113$&$2\t73\t113 $\\
$137$&$2\t137\t149 $\\
$149$&$5\t149 $\\
$157$&$17\t157 $\\
$173$&$173 $\\
$181$&$97\t181 $\\
$193$&$10\t109\t193\t233 $\\
$197$&$2\t197 $\\
$229$&$2\t173\t229 $\\
$233$&$2\t37\t41\t233 $ \\
$241$&$50\t241\t182969 $\\
$257$&$50\t41\t257 $\\
$269$&$10\t269 $\\
$277$&$157\t277 $\\
$281$&$10\t281\t12689 $\\
$293$&$293 $\\
$313$&$10\t17\t29\t313\t2909 $\\
$317$&$2\t317 $\\
$337$&$2\t337\t55335641 $\\
$349$&$2\t17\t29\t349 $\\
$353$&$2\t353\t3793 $ \\
$373$&$10\t53\t373 $\\
$389$&$10\t13\t389 $\\
$397$&$173\t397 $\\
\hline
\end{tabular}
\begin{tabular}{|c|c|}
\hline
\(p\) & $\Tr N(1-\zeta)$ \\ 
\hline
$401$&$2\t349\t401\t\t3749$\\
$569$&$10\t17\t569\t1427753$\\
$2909$&$5\cdot2909\t21084733$\\
$3793$&$2\t53\t3049\t3793\t9649\t96635049\t7180008526861$\\
$9649$&$10\t9649\t$\\
&$11038012868817426073312199225721320308683115045041507025035521109$\\[3pt]
$15313$&$10\t15313\t193674048965143013\t31697882584832654017\t$\\
&$1520919562609686936079944439228416723052033102529$\\[3pt]
$15329$&$50\t15329\t280450985814870077705875452193$\\[3pt]
$15349$&$10\t229\t653\t15349\t3083621\t41029529\t8411625925256297$\\
\hline
\end{tabular}

\subsection*{B. Lucas numbers of prime index.}

The following table lists Lucas numbers $L(p)$ in the form of product of pseudoprimes for  prime  numbers $p$ bigger than $3$. Note that, by the Theorem $1$ and the fact that $L(n)$ is even if and only if $3$ divides $n$,  all prime factors end by $1$ or $-1$. Note also that some of the last factorizations are long  to calculate on software like PARI.

\hfil
\begin{tabular}{|c|c|}
\hline
$p$ & $L(p)$ \\ 
\hline
$5$ & $11$\\
$7 $& $29$\\
$11 $&$ 199$\\
$13 $&$ 521$\\
$17 $&$ 3571$\\
$19 $&$ 9349$\\
$23 $&$ 139\cdot 461$\\
$29 $&$ 59\cdot 19489$\\
$31 $&$ 3010349$\\
$37 $&$ 54018521$\\
$41 $&$ 370248451$\\
$43 $&$ 6709\cdot 144481$\\
$47 $&$ 6643838879$\\
$53 $&$ 119218851371$\\
$59 $&$ 709\cdot 8969\cdot 336419$\\
$61 $&$ 5600748293801$\\
$67 $&$ 4021\cdot 24994118449$\\
$71 $&$ 688846502588399$\\
$73 $&$ 151549\cdot 11899937029$\\
$79 $&$ 32361122672259149$\\
$83 $&$ 35761381\cdot 6202401259$\\
$89 $&$ 179\cdot 22235502640988369$\\
$97 $&$ 3299\cdot 56678557502141579$\\
$101 $&$ 809\cdot 7879\cdot 201062946718741$\\
$103 $&$ 619\cdot 1031\cdot 5257480026438961$\\
$107 $&$ 47927441\cdot 479836483312919$\\
$109 $&$ 128621\cdot 788071\cdot 593985111211$\\
$113 $&$ 412670427844921037470771$\\
$127 $&$ 509\cdot 5081\cdot 487681\cdot 13822681\cdot 19954241$\\
$131 $&$ 1049\cdot 414988698461\cdot 5477332620091$\\
$137 $&$ 541721291\cdot 78982487870939058281$\\
$139 $&$ 30859\cdot 253279129\cdot 14331800109223159$\\
$149 $&$ 952111\cdot 4434539\cdot 3263039535803245519$\\
$151 $&$ 1511\cdot 109734721\cdot 217533000184835774779$\\
$157 $&$ 39980051\cdot 16188856575286517818849171$\\
$163 $&$ 1043201\cdot 6601501\cdot 1686454671192230445929$\\
$167 $&$ 766531\cdot 103849927693584542320127327909$\\
$173 $&$ 78889\cdot 6248069\cdot 16923049609\cdot 171246170261359$\\
$179 $&$ 359\cdot 1066737847220321\cdot 66932254279484647441$\\
$181 $&$ 97379\cdot 21373261504197751\cdot 32242356485644069$\\
$191 $&$ 22921\cdot 395586472506832921\cdot 910257559954057439$\\
$193 $&$ 303011\cdot 76225351\cdot 935527893146187207403151261$\\
$197 $&$ 31498587119111339\cdot 4701907222895068350249889$\\
\hline
\end{tabular}
\hfil

\hspace*{-3em} \begin{tabular}{|c|c|}
\hline
$p$ & $L(p)$ \\ 
\hline
$199 $&$ 2389\cdot 4503769\cdot 36036960414811969810787847118289$\\
$211 $&$ 33128448586319\cdot 3768695026320506495615952689771$\\
$223 $&$ 209621\cdot 191782505151874799799825102831271417475449$\\
$227 $&$ 39499\cdot 5098421\cdot 4311537234701\cdot 317351386961794678797301$\\
$229 $&$ 6871\cdot 104990418946773667410736999685208265866007631$\\
$233 $&$ 818757341\cdot 6911530261\cdot 873757179900549251563653697571$\\
$239 $&$ 479\cdot 7649\cdot 24216191671442408226762026802756956706931169$\\
$241 $&$ 1156801\cdot 4645999\cdot 43219877626484550971962471774087607599$\\
$251 $&$ 15061\cdot 170179\cdot 712841\cdot 15636705475517134545061743537722067281$\\
$257 $&$ 2107028233569599\cdot 125090447782502159\cdot 1945042261468790758531$\\
$263 $&$ 1579\cdot 924709\cdot 2098741\cdot 3001949101336686906107454320302466346629$\\
$269 $&$ 13451\cdot 49098524855733491\cdot 290341026883813109\cdot 860882346042166879$\\
$271 $&$ 59621\cdot 899179\cdot 92206663291\cdot 87426439096566323815478492553863521$\\
$277 $&$ 1109\cdot 5923369\cdot 1003666289\cdot 322458613167451\cdot 3647646099535497480264359$\\
$281 $&$ 20567460049\cdot 46415343154434259\cdot 55678135331080359350346681814561$\\
$283 $&$ 1699\cdot 252605941501\cdot 324238999448153864959724538289151678378314771$\\
$293 $&$ 287141\cdot 59605095029402530487010572214642235677583217188556211631$\\
$307 $&$ 1229\cdot 11739610117429203651282768407085324070169775523763828726810201$\\
$311 $&$ 34211\cdot 2890615644252691924572487628689034423952562309093965400390309$\\
$313 $&$ 258899611203303418721656157249445530046830073044201152332257717521$\\
$317 $&$ 4014648883841\cdot 15670596807846410359\cdot 28206477527834707033776102306507709$\\
$331 $&$ 526291\cdot 54184296181\cdot 4386848568249611\cdot 11957954590103942275063852978039182929$\\
$337 $&$ 21569\cdot 340819559\cdot 3651575156022459933890370204120436853816552382514920496951$\\
$347 $&$ 662771\cdot 25008386631867389\cdot 199187460399042526805980487374118125053459971811841$\\
$349 $&$ 81922033248592814089\cdot 15360894609285281651561\cdot 6868615650726652471695323898169$\\
$353 $&$ 59242995313457729780510823767354730798286848921481374874264534705573628371$\\
$359 $&$ 719\cdot 1648529\cdot 1517456267839\cdot 591045866085042506389105054361881124597940603179391791$\\
$367 $&$ 2298889\cdot 19997474011\cdot 28770822474564239\cdot 886000936153274021\cdot 42617146790676471298183229$\\
$373 $&$ 2239\cdot 400254035815230315691149177295200682735160941001088181098834904823746738439$\\
$379 $&$ 342912379\cdot 263849805823819\cdot 177736127099922813767360469853533109979011440647545873549$\\
$383 $&$ 7901291\cdot 359239103599\cdot 148183743669565231\cdot 262049896393105557884983790284397553466196701$\\
$389 $&$ 3904919893017807509\cdot 506500558078271291193254167726983153455094605128812878961431839$\\
$397 $&$ 138787200818838488796281\cdot 669489488461007623536076657264221325654746723636804695546641$\\
$401 $&$ 1775629\cdot 6238759\cdot 57490131422237119538947323064108369354763411739839193426140731363438441$\\
$409 $&$ 1427411\cdot 3165661\cdot 919046182779201475951\cdot 7204338025090306465053435447585274989066236336317481$\\
$419 $&$ 839\cdot 897499\cdot 316722762859\cdot 71770902070121337353161\cdot 214979797032908476941753114627650494795860641$\\
$421 $&$ 1995541\cdot 2971841379558728500308112219\cdot 1624464648249211515329083729905902386498198626045929519$\\
$431$& $188779 \cdot 203357221762049\cdot 1378049553112721035601841\cdot 223972569257357477422541897473184614027221489$\\
$433$& $12361187009\cdot 203421415129\cdot 110588112062920249\cdot 1115534320421853740681331360123414376301666570579820689$\\
$439$& $4391\cdot 826061911\cdot 1149884364774317448679684799063561\cdot 13345777072640252551741515140876418537140353109$\\
$443$& $2659\cdot  135559\cdot  1058468497584673231839683421107696696195973642200254474517777673097952198651301042859 $\\
$449$& $369079\cdot  72195777446499975249912346541\cdot  311680756181475991522861299701\cdot  824345650181758925250079965109$\\
$457$& $143392891\cdot  48175086409\cdot  864351271995241\cdot  53865562038701008975397146407705442118820462326130285905669299$\\
\hline
\end{tabular}

 



\subsection*{C. A source code to compute the Thue-Morse sequence.}
The following fragments of source codes are intended to calculate a value of the Thue-Morse sequence and they can be inserted into a C program. They use as input a variable \it n \rm of type \tt unsigned int \rm and a variable  \it thueMorse \rm of type \tt int\rm, and assign to \it thueMorse \rm the value
$$t_n=(-1)^\text{the number of $'1'$ digits in binary expansion of $n$}. $$

The following code is written in the AT\&T syntax and designed for commpilation by \tt gcc. 

\begin{verbatim}
asm("movl $1, %0;"
	"cmpl $0, %1;"
	"jnp _byte2;"
	"negl %0;"
"_byte2:;"
"	roll $8,%1;"
"	cmpl $0, %1;"
"	jnp _byte3;"
"	negl %0;"
"_byte3:;"
"	roll $8,%1;"
"	cmpl $0, %1;"
"	jnp _byte4;"
"	negl %0;"
"_byte4:;"
"	roll $8,%1;"
"	cmpl $0, %1;"
"	jnp _fini;"
"	negl %0;"
"_fini:;"
"roll $8,%1;"
:"=&r"(thueMorse)
:"r"(n)
	);
\end{verbatim}

\rm Users of Visual C++ will need the following version (written with INTEL syntax) instead:

\begin{verbatim}
__asm{push eax
	push ebx
	mov ebx, n

	mov eax, 1 
	cmp ebx, 0 
	jnp _byte2
	neg eax
_byte2:
	rol ebx,8
	cmp  ebx,0
	jnp _byte3
	neg eax
_byte3:
	rol ebx, 8
	cmp ebx, 0
	jnp _byte4
	neg eax
_byte4:
	rol ebx, 8
	cmp ebx, 0
	jnp _fini
	neg eax
_fini:

	mov thueMorse, eax
	pop ebx
	pop eax
	};
\end{verbatim}

These inline assembly codes outperform anything that can be written in C, which is why I suggest to use them as a fragment of code instead of function.

\end{document}